\documentclass[11pt]{amsart}

\usepackage[utf8]{inputenc}
\usepackage[T1]{fontenc}
\usepackage{lmodern}

\usepackage{longtable}
\usepackage{multicol,footnote}
\usepackage{bm,mathrsfs}
\usepackage{fancyvrb}
\usepackage[colorlinks=true,citecolor=blue,hypertexnames=false]{hyperref}
\usepackage{color, comment}

\usepackage{url,fullpage}
\usepackage{longtable}
\usepackage{multicol}

\usepackage{cite}

\usepackage{amsmath,amssymb}

\usepackage{version}

\usepackage{url}

\usepackage{graphicx}

\newcommand{\x}{x}

\newcommand{\Dx}{\partial}

\newcommand{\ie}{{\it i.e.}}
\newcommand{\eg}{{\it e.g.}}

\newcommand{\Cov}{\operatorname{Cov}}

\def\hoeven#1{{#1}}

\newcommand{\EE}{\mathbb{E}}

\def\Q{\mathbb Q}

\newtheorem{thm}{Theorem}

\newtheorem{cor}[thm]{Corollary}

\def\<#1>{\langle#1\rangle}

\def\deg{\operatorname{deg}}

\def\threeFtwo#1#2#3#4#5#6{{_3F_2}\biggl(\begin{matrix}
  {#1}\kern.707em {#2}\kern.707em{#3}\\{#4}\kern1em{#5}
\end{matrix}\,\bigg|\,#6\biggr)}
\def\twoFone#1#2#3#4{{_2F_1}\biggl(\begin{matrix}
  {#1}\kern.707em {#2}\\{#3}
\end{matrix}\,\bigg|\,#4\biggr)}

\def\stepset#1#2#3#4#5#6#7#8{%
  \begin{picture}(20,20)(-10,-10)
    \put(0,0){\ifx1#1\thicklines\let\x\vector\else\thinlines\let\x\line\fi\x(-1,-1){10}}
    \put(0,0){\ifx1#2\thicklines\let\x\vector\else\thinlines\let\x\line\fi\x(0,-1){10}}
    
\put(0,0){\ifx1#3\thicklines\let\x\vector\else\thinlines\let\x\line\fi\x(1,-1){10}}

\put(0,0){\ifx1#4\thicklines\let\x\vector\else\thinlines\let\x\line\fi\x(-1,0){10}}

\put(0,0){\ifx1#5\thicklines\let\x\vector\else\thinlines\let\x\line\fi\x(1,0){10}}
    \put(0,0){\ifx1#6\thicklines\let\x\vector\else\thinlines\let\x\line\fi\x(-1,1){10}}
    \put(0,0){\ifx1#7\thicklines\let\x\vector\else\thinlines\let\x\line\fi\x(0,1){10}}
           \put(0,0){\ifx1#8\thicklines\let\x\vector\else\thinlines\let\x\line\fi\x(1,1){10}}
  \end{picture}}

\newcounter{mytable}\def\mytable{\noindent\refstepcounter{mytable}\textbf{Table~\arabic{mytable}}}

\title[Non-D-finite excursions in the quarter plane]{Non-D-finite excursions in the quarter plane}
\author[Alin Bostan]{Alin Bostan}
\address{INRIA (France)}
\email{Alin.Bostan@inria.fr}
\author[Kilian Raschel]{Kilian Raschel}
\address{CNRS \& F\'ed\'eration Denis Poisson \& Laboratoire de Math\'ematiques et Physique Th\'eorique (France)}
\email{Kilian.Raschel@lmpt.univ-tours.fr}
\author[Bruno Salvy]{Bruno Salvy}
\address{INRIA (France)}
\email{Bruno.Salvy@inria.fr}

\date{\today}
\begin{document}

\begin{abstract}The number of excursions (finite paths starting and ending at the origin) having a given number of steps and obeying various geometric constraints is a classical topic of combinatorics and probability theory. 
We prove that the sequence $(e^{\mathfrak{S}}_n)_{n\geq 0}$ of numbers of excursions in the quarter plane corresponding to a nonsingular step set~$\mathfrak{S} \subseteq \{0,\pm 1 \}^2$ with infinite group does not satisfy any nontrivial linear recurrence with polynomial coefficients. Accordingly, in those cases, the trivariate generating function of the numbers of walks with given length and prescribed ending point is not D-finite. Moreover, we display the asymptotics of $e^{\mathfrak{S}}_n$. 
\end{abstract}

\maketitle 

\section{Introduction} 

\subsection{General context} Counting walks in a fixed region of the lattice
$\mathbb{Z}^d$ is a classical problem in probability theory~\cite{Spitzer64,Feller68,DoSn84,Cohen92,FeFrSo92,Weiss94,Woess00,BoBo08,LaLi10,BlVo11} and in enumerative
combinatorics~\cite{Polya21,Narayana79,Mohanty79}.
In recent years, the case of walks restricted to the quarter plane
$\mathbb{N}^2 = \{ (i,j) \in \mathbb{Z}^2 \, | \, i\geq 0, \, j \geq 0 \} $
has received special attention, and much progress has been done on this
topic~\cite{Kreweras65,Maher78,GB86,GeZe92,FaIaMa99,BM02,BoPe03,Bousquet05,Mishna07,BoKa09,KaKoZe09,Mishna09,MiRe09,BoKa10,BoMi10,FaRa10,FaRa11,KuRa11,KuRa12,FaRa12,Raschel12,MeMi13,DeWa13,BoKuRa13,JoMiYe13,BCHKP}.
Given a {finite} set $\mathfrak{S}$ of allowed steps, the general problem is to study
$\mathfrak{S}$-walks in the quarter plane~$\mathbb{N}^2$, that is walks
confined to~$\mathbb{N}^2$, starting at $(0,0)$ and using steps
in~$\mathfrak{S}$ only. Denoting by $f_{\mathfrak{S}}(i,j,n)$ the number of
such walks that end at~$(i,j)$ and use exactly~$n$ steps, the main high-level
objective is to \emph{understand} the generating function \[
F_{\mathfrak{S}}(x,y,t)=\sum_{i,j,n\geq 0} f_{\mathfrak{S}}(i,j,n)x^i y^j t^n
\; \in \mathbb{Q}[[x,y,t]], \] since this \emph{continuous} object captures
a great amount of interesting combinatorial information about the
\emph{discrete} object $f_{\mathfrak{S}}(i,j,n)$. For instance, the
specialization $F_{\mathfrak{S}}(1,1,t)$ is the generating function of the numbers of ${\mathfrak{S}}$-walks with
prescribed length, the specialization $F_{\mathfrak{S}}(1,0,t)$ is that of
${\mathfrak{S}}$-walks ending on the horizontal axis, and the specialization
$F_{\mathfrak{S}}(0,0,t)$ counts ${\mathfrak{S}}$-walks returning to the
origin, called ${\mathfrak{S}}$-\emph{excursions}.

\subsection{Questions} From the combinatorial point of view, the ideal goal
would be to find a closed-form expression for $f_{\mathfrak{S}}(i,j,n)$, or at
least for $F_{\mathfrak{S}}(x,y,t)$. This is not possible in general, even if
one restricts to particular step sets~$\mathfrak{S}$. Therefore, it is
customary to address more modest, still challenging, questions such as: 
{What is the asymptotic behavior of the sequence $f_{\mathfrak{S}}(i,j,n)$?
What are the structural properties of $F_\mathfrak{S}(x,y,t)$: is it} rational?
is it \emph{algebraic}\footnote{That is, root of a polynomial in $\mathbb{Q}[x,y,t,T]$.}?
or more generally
 \emph{D-finite}\footnote{In one variable~$t$ this means solution of a linear differential equation with coefficients in~$\mathbb{Q}[t]$; in several variables the appropriate generalization~\cite{Lipshitz1989} is that the set of all partial derivatives spans a finite-dimensional vector space over~$\mathbb{Q}(x,y,t)$.}? These questions are
related, since the asymptotic behavior of the coefficient sequence of a power
series is well understood for rational, algebraic and D-finite power series~\cite[Part~B]{FlSe09}.

\subsection{Main result} In this work, we prove that the generating function
\[F_{\mathfrak{S}}(0,0,t) = \sum_{n \geq 0} e^{\mathfrak{S}}_n t^n \; \in
\mathbb{Q}[[t]]\] of the sequence $(e^{\mathfrak{S}}_n)_{n\geq 0}$ of
${\mathfrak{S}}$-excursions \emph{is not D-finite} for a large class of
walks in the quarter plane. Precisely, this large class corresponds to all
\emph{small-step} sets $\mathfrak{S} \subseteq \{0,\pm 1 \}^2{\setminus \{0,0\}}$ for which a certain group $G_{\mathfrak{S}}$ of birational
transformations \emph{is infinite}, with the exception of a few cases for
which $F_{\mathfrak{S}}(0,0,t) = 1$ is trivially D-finite (these exceptional
cases are called \emph{singular}, see below). If $\chi = \chi_{\mathfrak{S}}$
denotes the \emph{characteristic polynomial} of the step set~$\mathfrak{S}$
defined by \[\chi(x,y) = \sum_{(i,j)\in\mathfrak{S}}x^{i}y^{j} \; \in
\mathbb{Q}[x,x^{-1},y,y^{-1}],\] then the group $G_{\mathfrak{S}}$ is
defined~\cite{FaIaMa99,BoMi10} as a group of rational automorphisms of
$\mathbb{Q}(x,y)$ that leave invariant the (Laurent) polynomial $\chi(x,y)$.
Up to some equivalence relations, there are 74 cases of nonsingular step sets
in $\mathbb{N}^2$, out of which 51 cases have an infinite group~\cite{BoMi10}. 
These 51 cases are depicted in Table~\ref{tab:2d} in
Appendix~\ref{sec:appendix}. With these definitions, our main result can be
stated as follows.

\begin{thm}\label{theo:main} Let $\mathfrak{S} \subseteq \{0,\pm 1 \}^2$ be
any of the 51 nonsingular step sets in $\mathbb{N}^2$ with infinite group~$G_{\mathfrak{S}}$. Then the generating function $F_{\mathfrak{S}}(0,0,t)$ of
${\mathfrak{S}}$-excursions is not D-finite. Equivalently, the excursion sequence $(e^{\mathfrak{S}}_n)_{n\geq 0}$ does not satisfy
any nontrivial linear recurrence with polynomial coefficients. In
particular, the full generating function $F_{\mathfrak{S}}(x,y,t)$ is not
D-finite. \end{thm}

By combining Theorem~\ref{theo:main} with previously known results, we obtain
the following characterization of nonsingular small-step sets with
D-finite generating function.
\begin{cor}\label{coro:main}
Let $\mathfrak{S} \subseteq \{0,\pm 1 \}^2$ be
	any of the 74 nonsingular step sets in $\mathbb{N}^2$.
The following assertions are equivalent:
\begin{enumerate}
	\item[(1)]  The full generating function $F_\mathfrak{S}(x,y,t)$ is D-finite;
	\item[(2)] the generating function $F_\mathfrak{S}(0,0,t)$ of ${\mathfrak{S}}$-excursions is D-finite;
	\item[(3)] the excursion sequence $e^{\mathfrak{S}}_{2n}$ is  asymptotically equivalent to $K \cdot \rho^n \cdot n^\alpha$, for some $\alpha \in \mathbb{Q}$; 
	\item[(4)] the group~${G}_{\mathfrak{S}}$ is finite;
	\item[(5)] the step set $\mathfrak{S}$ has either an axial symmetry, or a zero drift and cardinality different from~5.
\end{enumerate} 
Moreover, under \emph{(1)--(5)}, the cardinality of ${G}_{\mathfrak{S}}$ is
equal to $ 2 \cdot \min\left\{ \ell \in
\mathbb{N}^{\star} \, | \, \frac{\ell}{\alpha + 1} \in \mathbb{Z}\right\} $.
\end{cor}
Implication $(4)\Rightarrow (1)$ is a consequence of results in~\cite{BoMi10,BoKa10},
proofs of $(1)\Rightarrow (2) \Rightarrow (3) \Rightarrow (4)$ are given
in the present article, and the equivalence of (2) and (5) is read off the
tables in Appendix~\ref{sec:appendix}. Condition~(5) seems unnatural, its purpose is to eliminate the three rotations of the ``scarecrow'' walk with step sets depicted in Figure~\ref{fig-scarecrow}, which have zero drift and non-D-finite generating functions.
The observation on the cardinality seems new and interesting. It can be checked from the data~\cite[Tables~1--3]{BoMi10}.

\begin{figure}
\centerline{\includegraphics[height=3cm]{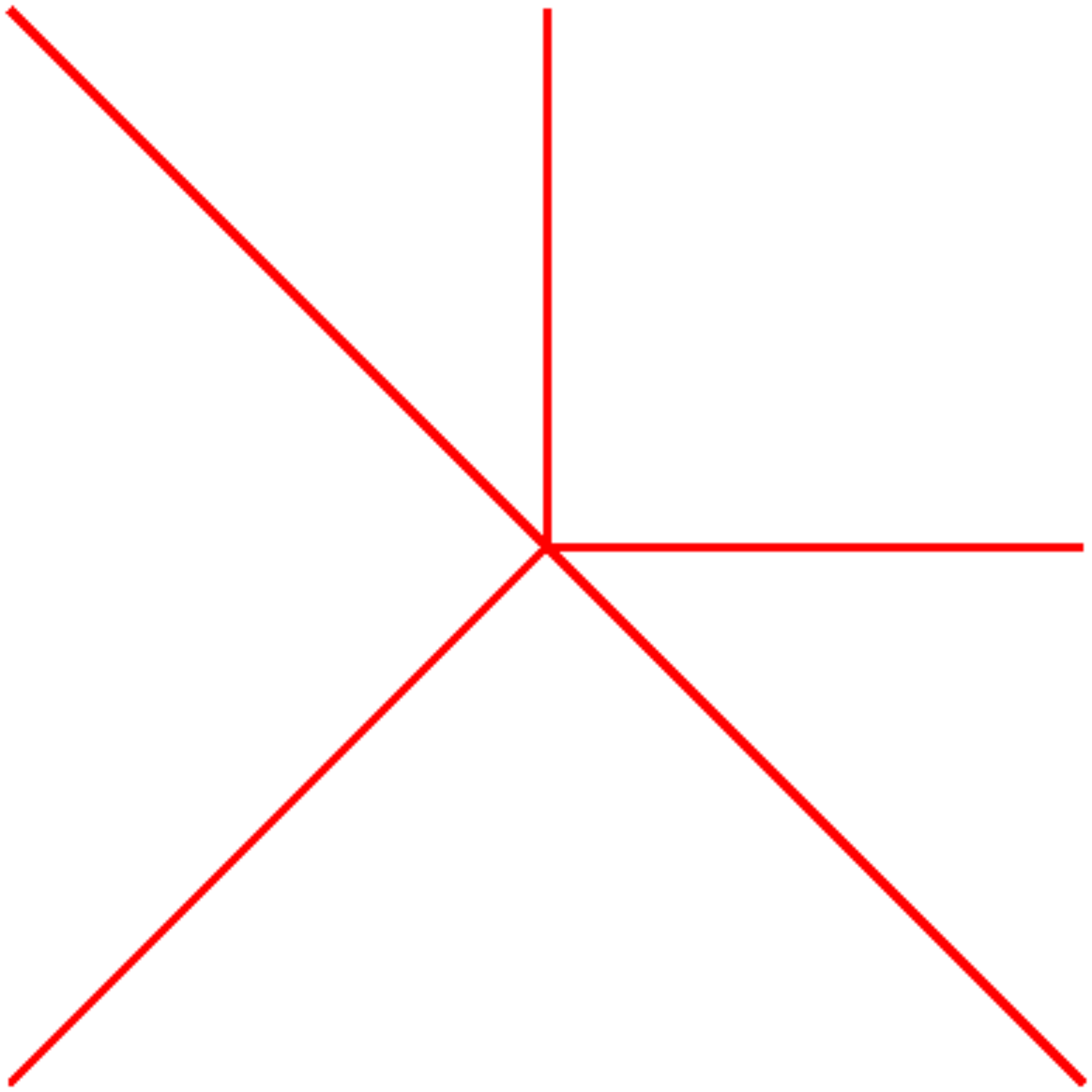}\quad\includegraphics[height=3cm]{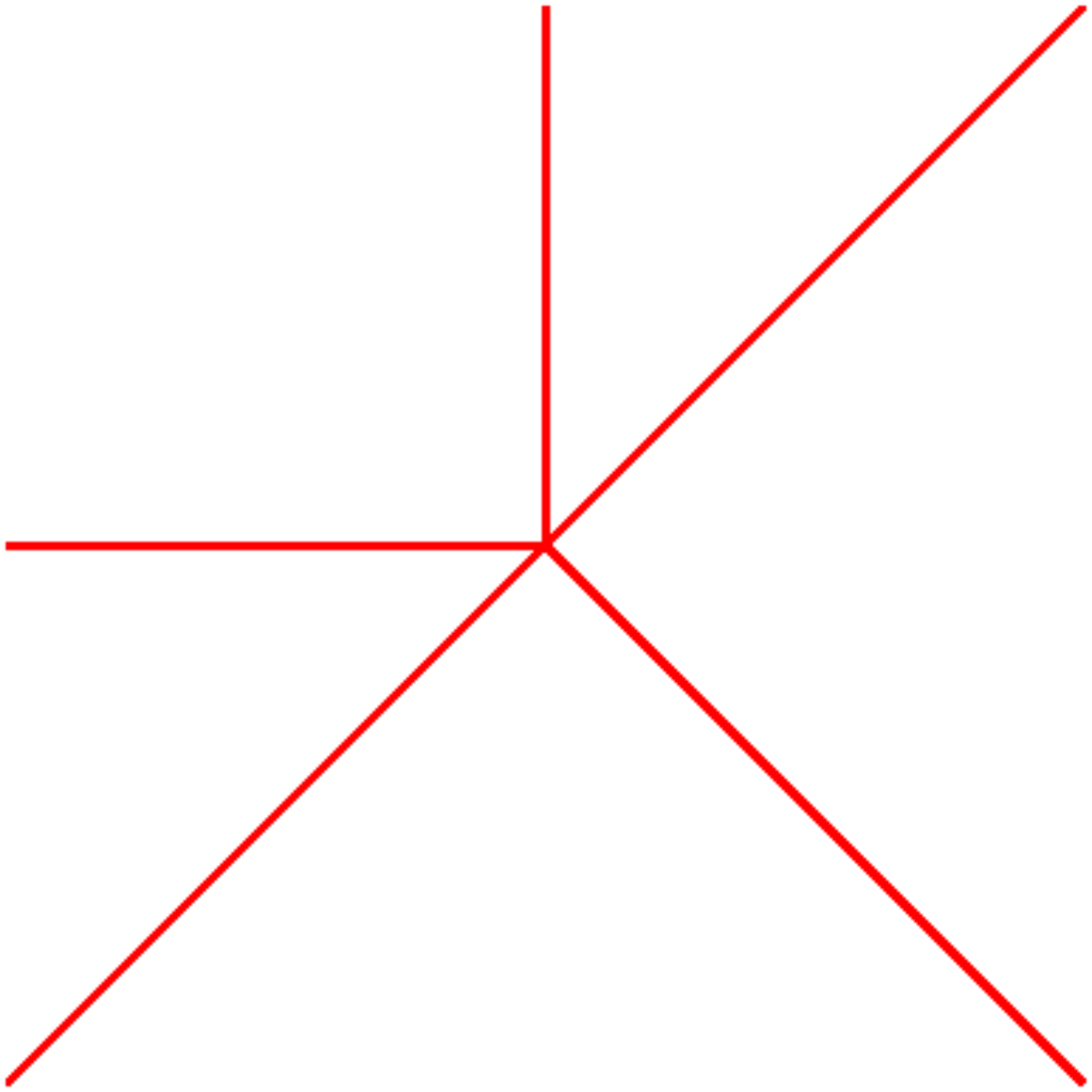}\quad\includegraphics[height=3cm]{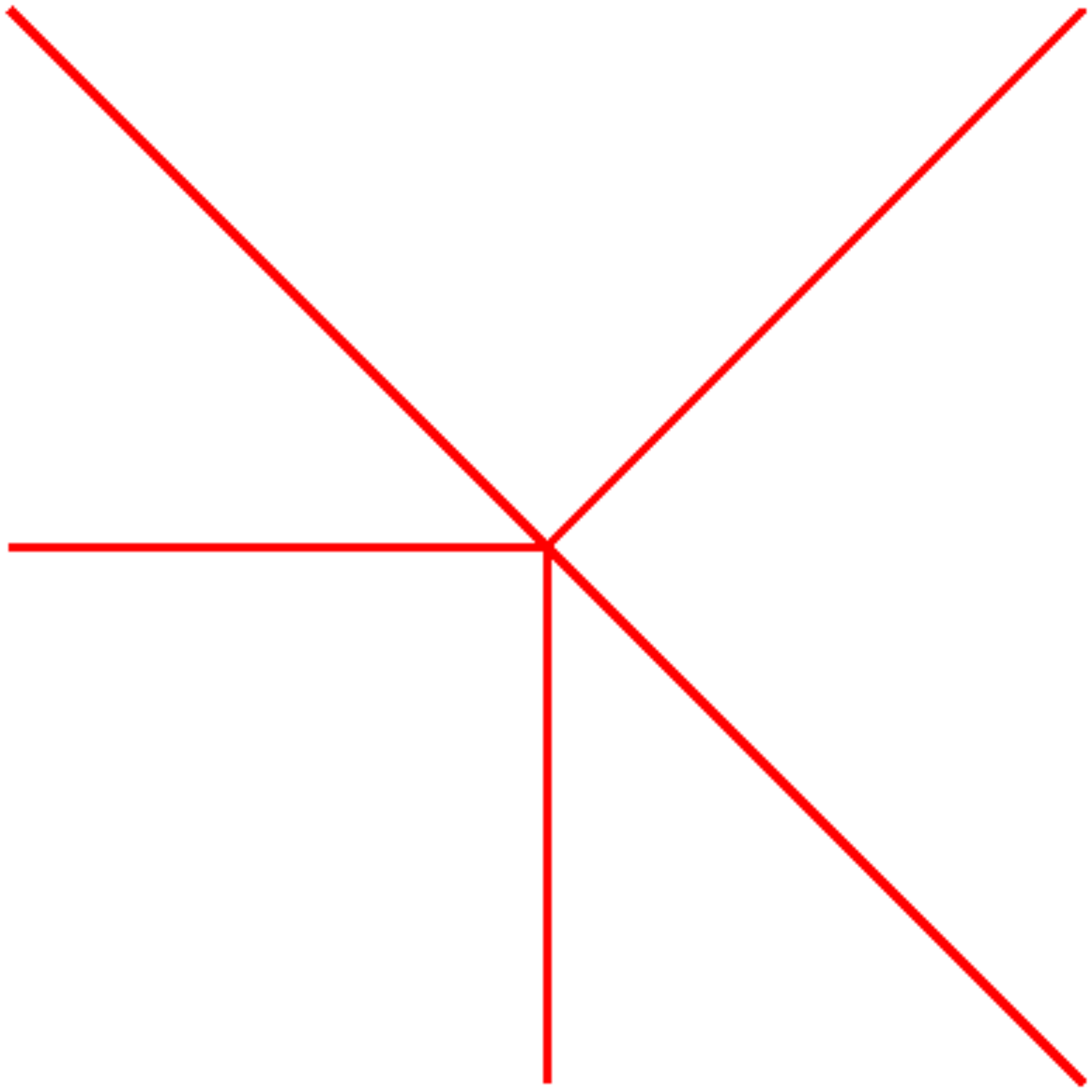}}
\caption{Rotations of a scarecrow. These are the three step sets (30, 40, 42 in Table~\ref{tab:2d}) of walks with zero drift that have a non-D-finite generating function.\label{fig-scarecrow}}
\end{figure}

\subsection{Previous results}
\subsubsection{Structural properties}
While it is known that unrestricted walks in $\mathbb{Z}^2$ have rational
generating functions and walks restricted to a half-plane in $\mathbb{Z}^2$
have algebraic generating functions (see~\cite{BaFl02} and~\cite[Proposition 2]{BoPe03}), a first intriguing result
about walks in the quarter plane is that their generating functions need not
be algebraic, and not even D-finite. For instance, Bousquet-M\'elou  and
Petkov{\v{s}}ek~\cite{BoPe03} proved that this is the case for the so-called
\emph{knight walk} starting at~$(1,1)$ with $\mathfrak{S}= \{(2, -1), (-1, 2)\}$. (Although this is not a walk with unit steps, it actually constitutes one of the initial motivations for the study of walks evolving in
the quarter plane.) It was later shown by Mishna and Rechnitzer~\cite{MiRe09} that 
even if one restricts to \emph{small-step walks}, where $\mathfrak{S} \subseteq \{0,\pm 1 \}^2$, there exist
step sets, such as $\mathfrak{S}=\{(-1,1),(1,1),(1,-1)\}$, for which the series
$F_{\mathfrak{S}}(x,y,t)$ is not D-finite (but in those cases~$F_{\mathfrak{S}}(0,0,t)=1$ is trivially D-finite).

In the remainder of this text, we restrict ourselves to small-step walks in
the quarter plane. Several sporadic cases of small-step walks with D-finite
generating functions have been known for a long time; the most famous ones are
Kreweras' walks~\cite{Kreweras65,Bousquet05}, with
$\mathfrak{S}=\{(-1,0),(0,-1),(1,1)\}$, and Gouyou-Beauchamps'
walks~\cite{GB86}, with $\mathfrak{S}=$$\{(1,0), (-1,0), (-1,1), (1,-1) \}$. A
whole class of small-step walks with D-finite generating functions was first
identified in~\cite[\S 2]{BoPe03}: this class contains step sets that admit an
axial symmetry. Another class, including $\mathfrak{S}=\{(0,1), (-1,0), (1,-1)
\}$ and $\mathfrak{S}=\{(0,1), (-1,0), (1,-1), (0,-1), (1,0), (-1,1) \}$,
corresponds to step sets that are left invariant by a Weyl group and whose
walks are confined to a corresponding Weyl chamber~\cite{GeZe92}.

A systematic classification of small-step walks with respect to the
D-finiteness of $F_{\mathfrak{S}}(1,1,t)$ was then undertaken by Mishna~\cite{Mishna07,Mishna09}
for step sets of cardinality at most three. A complete, still conjectural,
classification without this restriction was obtained by Bostan and Kauers~\cite{BoKa09} using computer algebra tools.
Almost simultaneously, Bousquet-M\'elou and Mishna~\cite{BoMi10} 
proved that among the~$2^8$ possible cases of small-step walks in the quarter plane, there are
exactly 79 inherently different cases. Among these, they
identified 22 cases of step sets $\mathfrak{S}$ having a D-finite
generating function $F_{\mathfrak{S}}(x,y,t)$.

A 23rd case, namely $\mathfrak{G} = \{(1,0), (-1,0), (1,1),
(-1,-1) \}$, known as \emph{Gessel walks}, is notoriously challenging. Its
generating function for excursions, $F_{\mathfrak{G}}(0,0,t)$, was first proved
to be D-finite by Kauers, Koutschan and Zeilberger~\cite{KaKoZe09}, using
computer algebra techniques. Then Bostan and Kauers~{\cite{BoKa10}} discovered
and proved that the full generating function $F_{\mathfrak{G}}(x,y,t)$ is
D-finite, and even algebraic, again using computer algebra.
It was proved afterwards by Fayolle and Raschel~\cite{FaRa10}
using a different approach that for any fixed value $t_0 \in (0,1/4)$, the
bivariate generating function $F_{\mathfrak{G}}(x,y,t_0)$ for Gessel walks is
algebraic over $\mathbb{R}(x,y)$, hence over~$\mathbb Q(x,y)$. Very recently, a ``purely human'' (\ie, computer-free) proof of the algebraicity of the full generating function
$F_{\mathfrak{G}}(x,y,t)$ was given by Bostan, Kurkova and Raschel~\cite{BoKuRa13}.

Bousquet-M\'elou and Mishna~\cite{BoMi10} showed that these 23 cases of step
sets $\mathfrak{S}$ with D-finite generating function
$F_{\mathfrak{S}}(x,y,t)$ correspond to walks possessing a finite group
$G_{\mathfrak{S}}$. Informally speaking, the \emph{group of a walk} is a
notion that captures symmetries of the step set and that is used to generalize
a classical technique in lattice combinatorics called the ``reflection
principle''~\cite[Ch.~III.1]{Feller68}. 

Moreover, it was conjectured
in~\cite{BoMi10} that the 56 remaining models with infinite group have
non-D-finite generating functions $F_{\mathfrak{S}}(x,y,t)$. This was proved
by Kurkova and Raschel~\cite{KuRa12} for the 51 \emph{nonsingular walks}, that
is, for walks having at least one step from the set $\{(-1,0), (-1,-1), (0,-1)
\}$. This result is obtained as a consequence of the non-D-finiteness of this
series as a function of $x,y$. We provide an alternative proof of this result by showing the non-D-finiteness of $F_{\mathfrak{S}}(0,0,t)$, about which nothing was known.

As for the singular walks, two out of the five cases
were already shown to have non-D-finite generating functions by Mishna and Rechnitzer~\cite{MiRe09}.
The last 3~cases of singular walks have a
generating function $F_{\mathfrak{S}}(x,y,t)$ that has just been proved to be non-D-finite by Melczer and Mishna~\cite{MeMi13}.

\subsubsection{Closed-form expressions} 
Closed forms are known for generating functions of the walks named after Kreweras~\cite{Kreweras65}, Gouyou-Beauchamps~\cite{GB86} and Gessel~\cite{KaKoZe09,BoKa10,BoKuRa13}.
Some other explicit formulas for
$f_{\mathfrak{S}}(i,j,n)$ and for $F_{\mathfrak{S}}(x,y,t)$, or some of their
specializations, have been obtained by Bousquet-Mélou and Mishna~\cite{BoMi10}
in cases when $\mathfrak{S}$ admits a finite group. A different type of
explicit expressions (integral representations) for the generating function of
Gessel walks was obtained by Kurkova and Raschel~\cite{KuRa11}. Their approach
was later generalized by Raschel~\cite{Raschel12} to all the 74 nonsingular
walks. Finally, Bostan, Chyzak, van Hoeij, Kauers and Pech~\cite{BCHKP}
used computer algebra tools to express all D-finite transcendental
functions $F_{\mathfrak{S}}(x,y,t)$ by iterated integrals of
Gaussian hypergeometric functions.

\subsubsection{Asymptotics} Concerning asymptotics, conjectural results were
given by Bostan and Kauers \cite{BoKa09} for the coefficients of
$F_{\mathfrak{S}}(1,1,t)$ when this function is D-finite. Some of these
conjectures have been proved by Bousquet-Mélou and Mishna~\cite{BoMi10}.
Explicit asymptotics for the coefficients of $F_{\mathfrak{S}}(0,0,t)$ and
$F_{\mathfrak{S}}(1,1,t)$ were conjectured even in non-D-finite cases in some
unpublished tables~\cite{BoKa}. In a recent work, Denisov and
Wachtel~\cite{DeWa13} have obtained explicit expressions for the asymptotics
of excursions $F_{\mathfrak{S}}(0,0,t)$ in a much broader setting; in
particular, their results provide (up to a constant) the dominating term in
the asymptotics of the $n$-th coefficient of $F_{\mathfrak{S}}(0,0,t)$ in
terms of the step set. Even more recently, Fayolle and Raschel~\cite{FaRa12}
showed that the dominant singularities of $F_{\mathfrak{S}}(0,0,t)$,
$F_{\mathfrak{S}}(1,0,t)$ and $F_{\mathfrak{S}}(1,1,t)$ are algebraic numbers,
and announced more general and precise results about asymptotics of
coefficients of $F_{\mathfrak{S}}(0,0,t)$, $F_{\mathfrak{S}}(1,0,t)$ and
$F_{\mathfrak{S}}(1,1,t)$. Furthermore, a combinatorial approach is proposed
by Johnson, Mishna and Yeats~\cite{JoMiYe13} that allows to find tight bounds
on the dominant singularities of the latter generating functions.

\section{Number Theory, Probability and Algorithms}

\subsection{Contributions}
\label{subsec:contributions}
In the present work, we prove the non-D-finiteness of the generating function of
${\mathfrak{S}}$-excursions $F_{\mathfrak{S}}(0,0,t)$ for all 51 cases of nonsingular walks
with infinite group. As a corollary, we deduce the non-D-finiteness of the
full generating function $F_{\mathfrak{S}}(x,y,t)$ for those cases since D-finiteness is preserved by specialization at~(0,0)~\cite{Lipshitz1989}. 
This
corollary has been already obtained by Kurkova and Raschel~\cite{KuRa12}, but the approach here is
at the same time simpler, and delivers a more accurate information. This new
proof only uses asymptotic information about the coefficients of
$F_{\mathfrak{S}}(0,0,t)$, and arithmetic information about the constrained
behavior of the asymptotics of these coefficients when their generating function is D-finite. More precisely,
we first make explicit consequences of the general results by Denisov and Wachtel~\cite{DeWa13} in
the case of walks in the quarter plane. If $e_n = e^{\mathfrak{S}}_n$ denotes
the number of ${\mathfrak{S}}$-excursions of length $n$ using only steps in~$\mathfrak{S}$,
this analysis implies that, when $n$ tends to infinity, $e_n$ behaves like $K
\cdot \rho^n \cdot n^\alpha$, where $K= K(\mathfrak{S})>0$ is a real number,
$\rho= \rho (\mathfrak{S})$ is an algebraic number, and $\alpha =
\alpha(\mathfrak{S})$ is a real number such that $c =-
\cos(\frac{\pi}{1+\alpha})$ is an algebraic number. 
Explicit real approximations for $\rho$, $\alpha$ and $c$ can be determined to arbitrary precision.
Moreover, exact minimal polynomials of $\rho$ and $c$ can be determined
algorithmically starting from the step set~$\mathfrak{S}$. For the 51 cases of
nonsingular walks with infinite group, this enables us to prove that the
constant $\alpha=\alpha(\mathfrak{S})$ is not a rational number. The proof
amounts to checking that some explicit polynomials in $\mathbb{Q}[t]$ are not
cyclotomic. To conclude, we use a classical result in the arithmetic theory of
linear differential equations~\cite{DwGeSu94,Andre00,Garoufalidis08} about
the possible asymptotic behavior of an integer-valued, exponentially bounded D-finite
sequence, stating that if such a sequence grows like $K \cdot \rho^n \cdot
n^\alpha$, then $\alpha$ is necessarily a \emph{rational number}. 

\begin{quote}\sl In summary, our approach brings together (consequences of) a
strong probabilistic result~\cite{DeWa13} and a strong arithmetic
result~\cite{DwGeSu94,Andre00,Garoufalidis08}, and demonstrates that this combination
allows for the \emph{algorithmic} certification of the non-D-finiteness of
the generating function of ${\mathfrak{S}}$-excursions $F_{\mathfrak{S}}(0,0,t)$ in the 51 cases
of nonsingular small-step walks with infinite group.
\end{quote}

\subsection{Number theory}
\label{sec:number_theory}

It is classical that, in many cases, transcendence of a complex function
can be recognized by simply looking at the local behavior around its
singularities, or equivalently at the asymptotic behavior of its Taylor
coefficients. This is a consequence of the Newton-Puiseux theorem and of
transfer theorems based on Cauchy's integral formula, see, \eg, \cite[\S
3]{Flajolet87} and~\cite[Ch.~VII.7]{FlSe09}. For instance, if $(a_n)_{n \geq
0}$ is a sequence whose asymptotic behavior has the form $K \cdot \rho^n
\cdot n^\alpha$ where either the \emph{growth constant}~$\rho$ is
transcendental, or the \emph{singular exponent}~$\alpha$ is irrational or a
negative integer, then the generating function $A(t) = \sum_{n \geq 0} a_n t^n$
is not algebraic.

A direct application of this criterion and our result on the irrationality of~$\alpha$ allows to show
that the generating function for ${\mathfrak{S}}$-excursions in the 51 cases of nonsingular
walks with infinite group is \emph{transcendental}. (This is of course also a consequence of their being non-D-finite.)

Similar (stronger, though less known) results, originating from the
arithmetic theory of linear differential equations, also allow to detect
non-D-finiteness of power series by using asymptotics of their coefficients. This is
a consequence of the theory of G-functions~\cite{Andre89,DwGeSu94}, introduced
by Siegel almost a century ago in his work on diophantine approximations~\cite{Siegel1929}.

We will only use a corollary of this theory, which is well-suited to
applications in combinatorics. This result is more or less classical, but we could not find its exact statement in the literature.

\begin{thm} \label{theo:arithmetic} Let $(a_n)_{n \geq 0}$ be an
integer-valued sequence whose $n$-th term $a_n$ behaves asymptotically like $K \cdot \rho^n \cdot n^\alpha$, for some real constant $K>0$. If the
\emph{growth constant}~$\rho$ is transcendental, or if the \emph{singular
exponent}~$\alpha$ is irrational, then the generating function $A(t) = \sum_{n
\geq 0} a_n t^n$ is not D-finite. \end{thm}
Classical results by
Birkhoff-Trjitzinsky~\cite{BiTr32} and Turrittin~\cite{Turrittin55} imply that
if the $n$-th coefficient of a D-finite power series is asymptotic to $K \cdot \rho^n \cdot n^\alpha$, then $\rho$ and $\alpha$ are
necessarily algebraic numbers.

The difficult part of Theorem~\ref{theo:arithmetic} is that irrationality of
the singular exponent implies non-D-finiteness, \emph{under the integrality
assumption} on the coefficients. The only proof that we are aware of uses the
fact that any D-finite power series with integer-valued and exponentially
bounded coefficients is a \emph{G-function}. It relies on the combination of
several strong arithmetic results. First, the Chudnovsky-Andr\'e
theorem~\cite{ChCh85,Andre89} states that the minimal order linear differential
operator satisfied by a G-function is \emph{globally nilpotent}. Next, Katz's
theorem~\cite{Katz70} shows that the global nilpotence of a differential operator
implies that all of its singular points are \emph{regular singular} points
with \emph{rational exponents}.

We refer to~\cite{DwGeSu94} for more details on this topic, and
to~\cite{Garoufalidis08} for a brief and elementary account.

\subsection{Probability theory}
\label{subsec:Probability}
Consider a walk starting from the origin such that, at each unit time, a jump is chosen uniformly at random in~$\mathfrak S$, independently of the previous steps. Let then~$\tau$ denote the first time when the boundary of the translated positive quarter plane $(\{-1\}\cup \mathbb N)^2$ is reached. If $(X_1(k),X_2(k))_{k\geq 1}$ denote the coordinates of the successive positions of the walk, then our enumeration problem is related to probability in a simple way:
\begin{equation}
\label{counting-probab}
\mathbb P\!\left[\sum_{k=1}^n (X_1(k),X_2(k)) = (i,j),\,\tau> n\right]=\frac{f_\mathfrak S(i,j,n)}{\vert\mathfrak S \vert^n}.
\end{equation}
With an appropriate scaling of both time and space, one gets, at first order, a continuous analogue of the walk, the Brownian motion. Using known results on the random walks in a cone and a refined analysis of the approximation by the Brownian motion, Denisov and Wachtel \cite{DeWa13} have obtained a precise asymptotic estimate of the probability in Eq.~\eqref{counting-probab}.

They make the hypothesis that the random walk is irreducible in the cone, which translates in our setting into the 
\emph{nondegeneracy} of the walk: for all $(i,j)\in\mathbb{N}^2$, the set~$\{n\in\mathbb{N}\mid f_{\mathfrak S}(i,j,n)\neq0\}$ is nonempty; furthermore, the walk is said to be \emph{aperiodic} when the gcd of the elements of this set is~1 for all $(i,j)$. Otherwise, it is \emph{periodic} and this gcd is the period.
We now state their result in a way that can be used directly in our computations.
\begin{thm}[Denisov \& Wachtel \cite{DeWa13}] \label{theo:proba}
Let $\mathfrak{S} \subseteq \{0,\pm 1 \}^2$ be the step set of a walk in the quarter plane~$\mathbb{N}^2$, which is not contained in a half-plane.

Let $e_n = e^{\mathfrak{S}}_n$ denote the number of ${\mathfrak{S}}$-excursions of length $n$ using only steps in~$\mathfrak{S}$, and 
let $\chi = \chi_{\mathfrak{S}}$ denote the characteristic polynomial $\sum_{(i,j)\in\mathfrak{S}}x^{i}y^{j} \in \mathbb{Q}[x,x^{-1},y,y^{-1}]$ of the step set~$\mathfrak{S}$. 
Then, the system 
\begin{equation}
\label{eq:system}
     \frac{\Dx \chi}{\Dx x} = \frac{\Dx \chi}{\Dx y} = 0
\end{equation} 
has a unique solution $(x_0,y_0) \in \mathbb{R}_{> 0}^2$.
Next, define
\begin{equation}
\label{eq:definition-c}
\rho := \chi(x_0,y_0),\qquad     c := \frac{\frac{\Dx^2 \chi}{\Dx x \Dx y}}{\sqrt{\frac{\Dx^2 \chi}{\Dx x^2} \cdot \frac{\Dx^2 \chi}{\Dx y^2}}}(x_0, y_0), \qquad \alpha := -1 - \pi/\arccos(-c).
\end{equation}
Then, there exists a constant~$K>0$, which depends only on~$\mathfrak{S}$, such that:
\begin{itemize}
\item[--] if the walk is aperiodic,
\[ e_n \sim K \cdot \rho^n \cdot n^\alpha, \] 
\item[--] if the walk is periodic (then of period $2$),
\[  e_{2n} \sim K \cdot \rho^{2n} \cdot (2n)^\alpha,\quad e_{2n+1} = 0.\] 
\end{itemize}
\end{thm}

\begin{proof}
The theorem is not stated explicitly under this form by Denisov and Wachtel. Following the discussion in their \S1.5, we now review how this result is a consequence of their theorems. Given a random walk~$X$ starting at the origin and with each step drawn uniformly at random in~$\mathfrak{S}$, the result is obtained by a succession of normalizations. These normalizations are illustrated in Figure~\ref{fig-walks} on the step set of Example~23 of Table~\ref{tab:2d}, namely $\mathfrak{S}=\{(-1,0),(0,1),(1,0),(1,-1),(0,-1)\}$. 

\begin{figure}
\includegraphics[height=4cm]{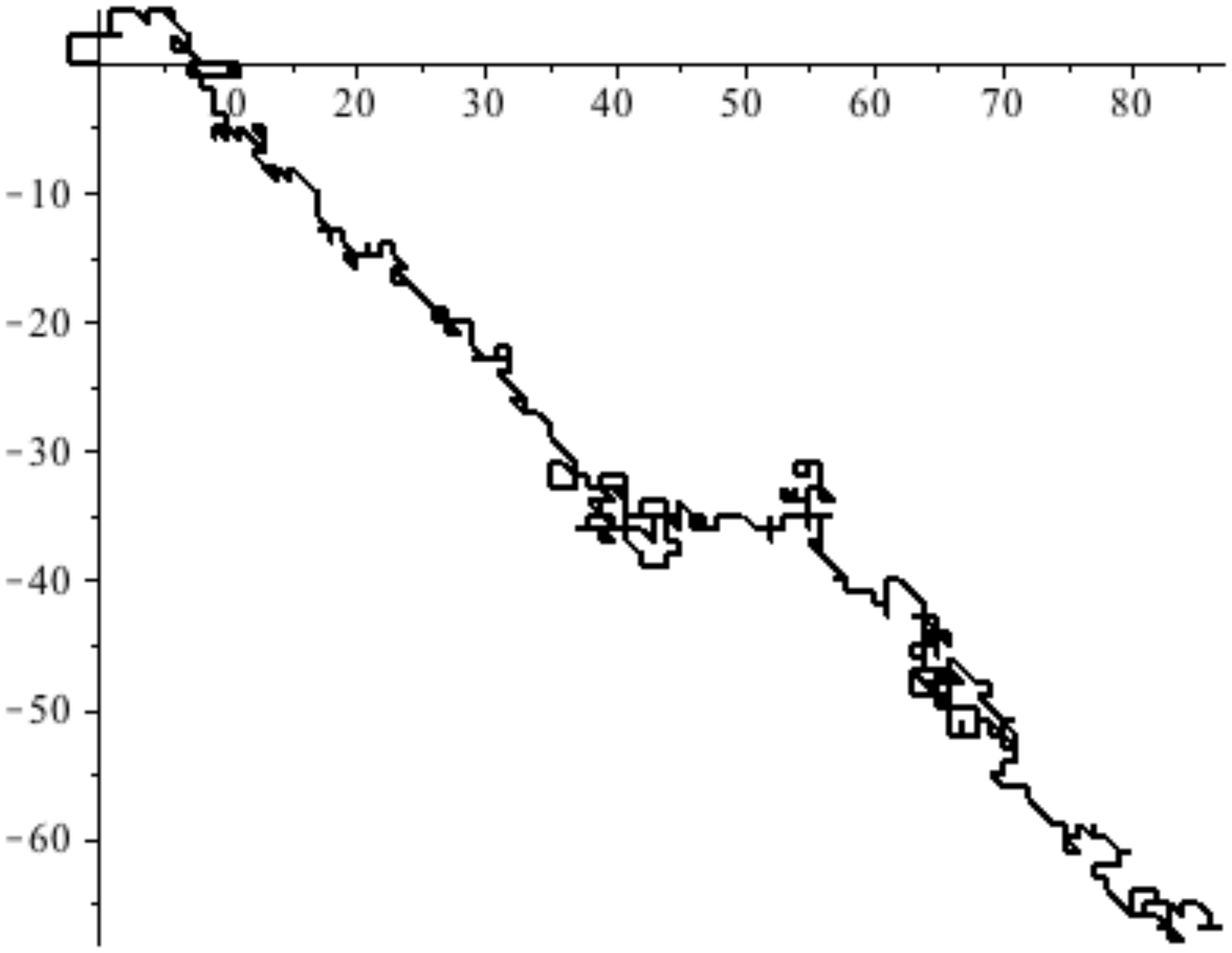}
\includegraphics[height=5cm]{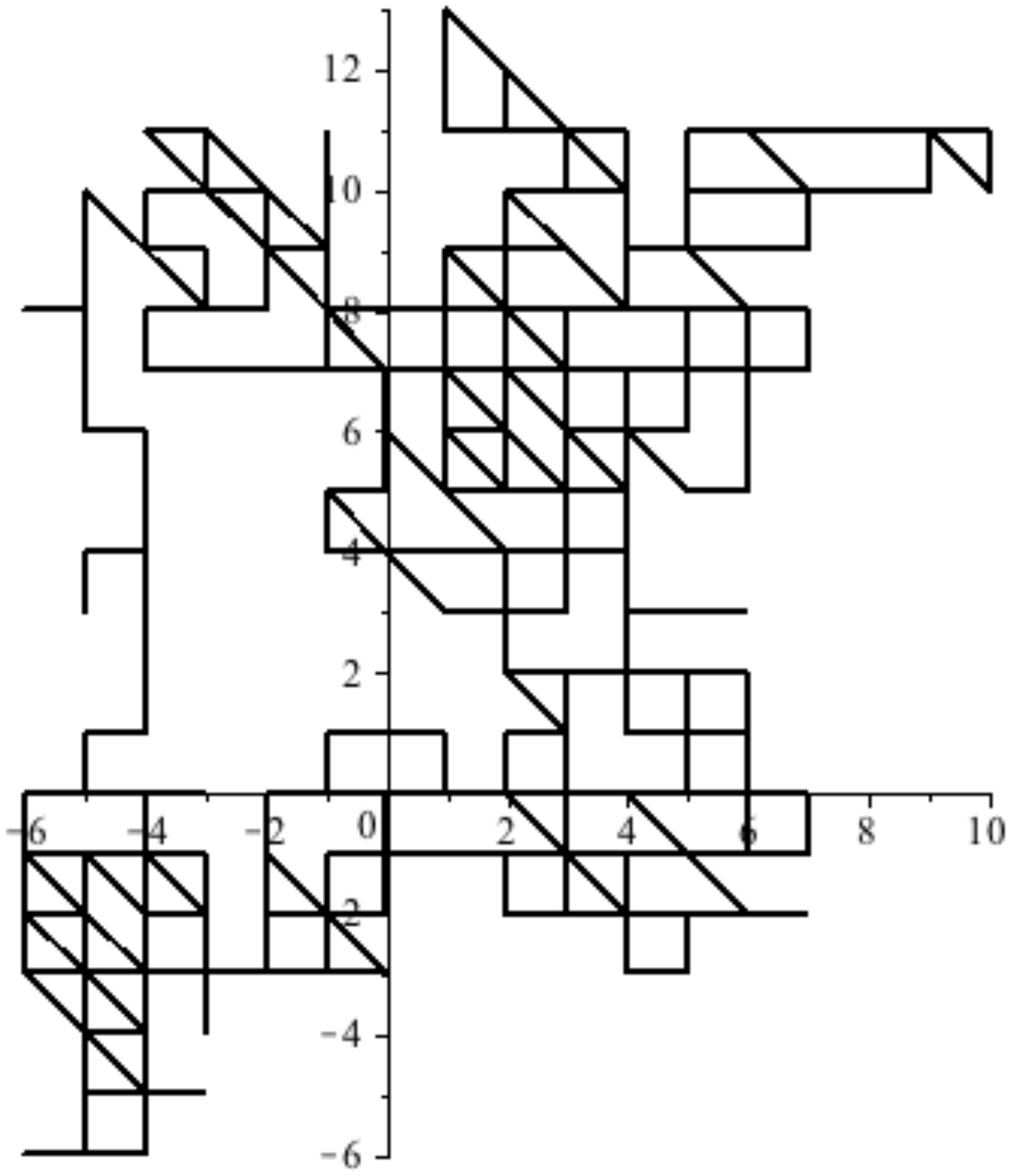}
\includegraphics[height=5cm]{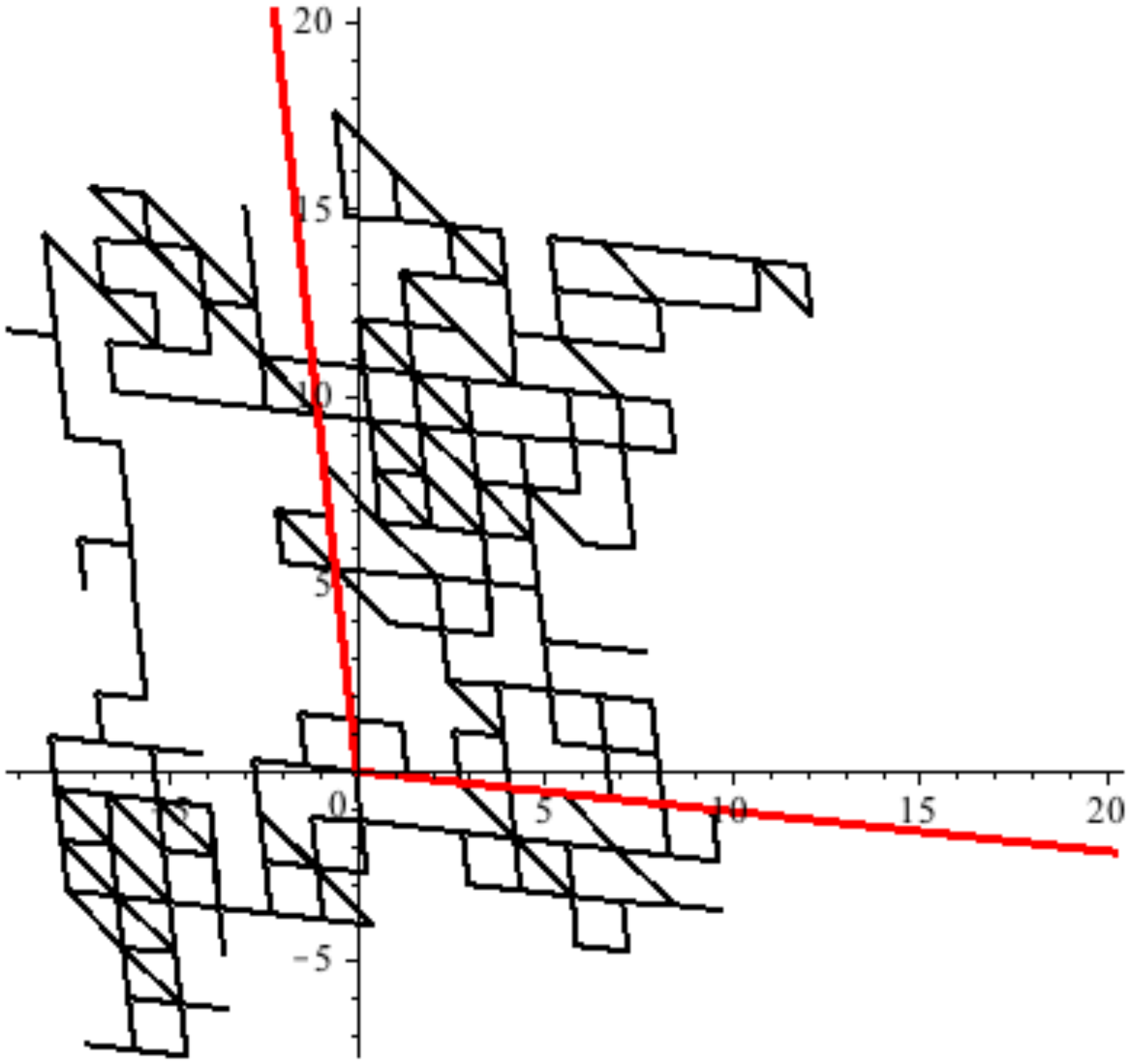}
\caption{Five hundred steps with $\mathfrak{S}=\{(-1,0),(1,0),(0,-1),(0,1),(1,-1)\}$.
{\small Random walk~$X$ with steps drawn uniformly from $\mathfrak{S}$ (left); random walk~$Y$ with steps drawn from $\mathfrak{S}$ with probabilities~$x_0^iy_0^j/\chi(x_0,y_0)$ (middle); random walk~$Z$ obtained by decorrelating~$Y$, with the cone $M(\mathbb N^2)$ (right).}}
\label{fig-walks}
\end{figure}
 
\medskip
\paragraph{\em Drift and weights} The first step is to reduce to the case of a random walk $Y$ with no drift (\ie, $\mathbb{E}[(Y_1(k),Y_2(k))]=(0,0)$ for all~$k$, where $Y_1$ and $Y_2$ are the coordinates of~$Y$).
This is achieved by giving different weights to each step: a weight~$x_0>0$ to the East direction, $1/x_0$ to the West direction and simultaneously~$y_0>0$ and~$1/y_0$ to the North and South directions. Each step~$(i,j)\in\mathfrak S$ is then given probability~$x_0^iy_0^j/\chi(x_0,y_0)$. 
Finally, $x_0$ and $y_0$ are fixed by the condition $\EE[(Y_1(k),Y_2(k))]=(0,0)$. (This is a special case of the Cram\'er transform, see \cite{Azencott1980}.) By differentiation with respect to~$x$ (resp.~$y$), the expectations are obtained as
\[\EE[Y_1(k)]=\frac{x_0}{\chi(x_0,y_0)}\frac{\partial\chi}{\partial x}(x_0,y_0),\quad
\EE[Y_2(k)]=\frac{y_0}{\chi(x_0,y_0)}\frac{\partial\chi}{\partial y}(x_0,y_0).
\]
A correct choice of~$(x_0,y_0)$ is therefore given by a positive solution to Eq.~\eqref{eq:system}.

Since the step set of the walk is not confined to the right half-plane, the limit of~$\chi(x,y)$ as $x\rightarrow0+$ is infinite, similarly for~$y\rightarrow0+$ and for $x$ or $y$ tending to~$+\infty$. This proves the existence of a solution. Its uniqueness comes from the convexity of~$\chi$, a Laurent polynomial with positive coefficients.

This new random walk~$Y$ is related to the original one: by induction on the number of steps,
\[\mathbb P\!\left[\sum_{k=1}^n (Y_1(k),Y_2(k)) = (i,j),\,\tau> n\right]=
x_0^iy_0^j\frac{|\mathfrak S|^n}{\chi(x_0,y_0)^n}\,
\mathbb P\!\left[\sum_{k=1}^n (X_1(k),X_2(k)) = (i,j),\,\tau> n\right],\]
where we use the same letter~$\tau$ to denote the exit times of~$X$ and~$Y$ from~$\mathbb{N}^2$.
In view of Eq.~\eqref{counting-probab}, the number of walks can be read off the new walk as
\[f_{\mathfrak S}(i,j,n)=\frac{\rho(x_0,y_0)^n}{x_0^i y_0^j}\,\mathbb P\!\left[\sum_{k=1}^n (Y_1(k),Y_2(k)) = (i,j),\,\tau> n\right].
\]

\medskip
\paragraph{\em Covariance and scaling} The second step is to reduce to the case of a random walk~$Z$ with no drift and no correlation, \ie, whose covariance matrix $\Cov(Z)=(\mathbb{E}[Z_iZ_j])_{i,j}$ is the identity matrix. 

The covariance matrix can be obtained from the characteristic polynomial again. Simple computations lead to
\[\Cov(Y)=\frac{1}{\chi(x_0,y_0)}\begin{pmatrix}x_0^2\frac{\partial^2\chi}{\partial x^2}(x_0,y_0)&x_0y_0\frac{\partial^2\chi}{\partial x\partial y}(x_0,y_0)\\
x_0y_0\frac{\partial^2\chi}{\partial x\partial y}(x_0,y_0)&y_0^2\frac{\partial^2\chi}{\partial y^2}(x_0,y_0)
\end{pmatrix}.
\]

One way to compute the appropriate scaling is in two steps. First, define a new walk obtained from~$Y$ by~$(W_1,W_2)=(Y_1/\sqrt{\mathbb{E}[Y_1^2]},Y_2/\sqrt{\mathbb{E}[Y_2^2]})$. By a direct computation, the walk~$W$ has no drift and satisfies
\[\mathbb{E}[W_1^2]=\mathbb{E}[W_2^2]=1,\quad
\mathbb{E}[W_1W_2]=\frac{\frac{\partial^2\chi}{\partial x\partial y}}
{\sqrt{\frac{\partial^2\chi}{\partial x^2}\frac{\partial^2\chi}{\partial y^2}}}(x_0,y_0)=c.\]
By the Cauchy-Schwarz inequality, the correlation coefficient~$c$ belongs to~$[-1,1]$.

Finally, an uncorrelated walk is obtained by modifying the directions of the steps, defining a new walk $Z=MW$. We must have $\text{Cov}(Z)$ equal to the identity matrix. Since $\text{Cov}(Z)=M\text{Cov}(W)M^\intercal$ and since $\text{Cov}(W)$, as a covariance matrix, is symmetric positive-definite, we can find a diagonal matrix $D$ with positive diagonal entries and an orthonormal matrix $P$ such that $ \text{Cov}(W)=PD{}P^\intercal$. The matrix $A=P{D}^{1/2}P^\intercal$ therefore satisfies $AA^\intercal=\text{Cov}(W)$, and the choice $M=A^{-1}=P{D}^{-1/2}P^\intercal$ is suitable.

In our case, we have
\begin{equation*}
     \text{Cov}(W)=\left(\begin{array}{cc}1&c\\c&1\end{array}\right),
     \quad P=P^\intercal=P^{-1}=\frac{1}{\sqrt{2}}\left(\begin{array}{cr}1&1\\1&-1\end{array}\right),
     \quad D=\left(\begin{array}{cc}{1+c}&0\\0&{1-c}\end{array}\right).
\end{equation*}
We deduce that
\begin{equation*}
     M=A^{-1}=P{D}^{-1/2}P=\frac{1}{2\sqrt{1-c^2}}\left(\begin{array}{cc}\sqrt{1+c}+\sqrt{1-c}&\sqrt{1-c}-\sqrt{1+c}\\\sqrt{1-c}-\sqrt{1+c}&\sqrt{1+c}+\sqrt{1+c}\end{array}\right).
\end{equation*}
Choosing $c=\sin(2\phi)$ and using easy trigonometric identities, we conclude that
\begin{equation*}
     M=\frac{1}{\sqrt{1-c^2}}\left(\begin{array}{cr}\cos(\phi)&-\sin(\phi)\\-\sin(\phi)&\cos(\phi)\end{array}\right).
\end{equation*}

Now, the excursions of these walks are related by
\begin{align*}
\mathbb P\!\left[\sum_{k=1}^n (Y_1(k),Y_2(k)) = (0,0),\,\tau> n\right]
&=\mathbb P\!\left[\sum_{k=1}^n (W_1(k),W_2(k)) = (0,0),\,\tau> n\right],\\
&=\mathbb P\!\left[\sum_{k=1}^n (Z_1(k),Z_2(k)) = (0,0),\,\tau> n\right],
\end{align*}
where the same letter $\tau$ is used to denote the first exit times, first for the walk~$Y$ from~$\mathbb{N}^2$, next for the walk~$W$ from~$\mathbb{N}^2$  and finally for the walk~$Z$ from the cone~$M(\mathbb{N}^2)$, whose opening is $\arccos(-\sin2\phi)=\arccos(-c).$ (See Figure~\ref{fig-walks}.)


\medskip
\paragraph{\em Asymptotic behavior of exit times}
In these conditions, the result of Denisov and Wachtel~\cite[Theorem~6]{DeWa13}
proves in great generality that the exit time behaves like the exit time of
the Brownian motion from that same cone. This is a classical topic of probability theory, in arbitrary dimension~\cite{Spitzer64}. We content ourselves with sketching how the exponent~$\pi/\arccos(-c)$ comes into play and refer to the literature for details. Our description follows closely that of DeBlassie~\cite{DeBlassie87} and Ba\~nuelos and Smits~\cite{BaSm97} that we make explicit in our special case.

The probability~$g(x,y,t)=\mathbb P_{(x,y)}[\tau\ge t]$ that a Brownian motion starting at~$(x,y)$ inside the cone is still inside the cone at time~$t$ obeys a diffusion equation
\[\left(\frac\partial{\partial t}-\frac{1}{2}\Delta\right)g(x,y,t)=0,\]
where $\Delta$ denotes the Laplacian, with $g(x,y,0)=1$ inside the cone and $g(x,y,t)=0$ for $t\ge 0$ on its border. Intuitively, this can be seen as the limit of the discrete recurrence for the probabilities of the random walk. It is then natural to pass to polar coordinates~$(r,\theta)$. 
By a classical homogeneity property of the Brownian motion, the solution is actually a function of~$t$ and~$s=t/r^2$, which implies an extra equation
$(\partial/{\partial t}+({r}/{2t})\partial/{\partial r})g(r,\theta,t)=0.$
Changing the variables into~$u(s,\theta)=g(r,\theta,t)$ finally leads to
\[\left(L_s+\frac{\partial^2}{\partial\theta^2}\right)u(s,\theta)=0,\quad\text{where}\quad L_s=s^2\frac{\partial^2}{\partial s^2}+2(2s-1)\frac\partial{\partial s}\]
with boundary conditions $u(0,\theta)=1$ for $\theta$ inside the cone and $u(s,\theta)=0$ for $s\ge0$ on its border. This problem is solved by the method of separation of variables: if a solution can be written $A(s)B(\theta)$, then 
$L_s(A(s))/A(s)=-B''(\theta)/B(\theta)$.
The left-hand side depends only on~$s$ and the right-hand side only on~$\theta$ and thus they are both equal to a constant~$\lambda$. In particular $B''(\theta)+\lambda B(\theta)=0$ with boundary conditions~$B(0)=B(\arccos(-c))=0$ forces $\lambda$ to be of the form~$\mu_k^2=(k\pi/\arccos(-c))^2$, $k\in\mathbb{N}\setminus\{0\}$ with corresponding solution~$\sin(\theta\mu_k)$. To each such~$\lambda_k=\mu_k^2$ corresponds a solution of the left-hand side in terms of a hypergeometric series ${}_1F_1$ (see~\cite[\S16.11]{NIST:DLMF}), namely
\[A_k(s)=(2s)^{-\mu_k/2}{}_1F_1\left(\mu_k/2,\mu_k+1,-1/(2s)\right),\quad
\text{with}\quad
\lim_{s\rightarrow0+}A_k(s)=\frac{2^{\mu_k}}{\sqrt{\pi}}\Gamma\!\left(\frac{\mu_k+1}{2}\right).\]
By completeness of the set of eigenfunctions (or Fourier expansion in that case), the solution of the diffusion equation therefore writes as a linear combination
\[g(r,\theta,t)=\sum{c_k\sin(\mu_k\theta)A_k(t/r^2)}.\]
The coefficients~$c_k$ are then given by the Fourier expansion of the boundary condition~$u(0,\theta)=1$.
Thus, finally, the desired probability~$g(x,y,t)$ has leading term 
in~$t^{-\mu_1/2}=t^{-\pi/(2\arccos(-c))}$ as $t\rightarrow\infty$.
The relation to the discrete random walk due to Denisov and Wachtel then gives 
\[\mathbb{P}_{(x,y)}[\tau\ge n]\sim \kappa n^{-\pi/(2\arccos(-c))}, \]
for some constant~$\kappa$ depending on~$x$ and~$y$. 

\medskip
\paragraph{\em Local limit theorem}
The conclusion is now explicitly in Denisov and Wachtel~\cite[Theorem~6 and Lemma~13]{DeWa13}.

\medskip
\paragraph{\em Periodic step sets}  An examination of the different cases of periodic step sets (they are marked with a star in Tables \ref{tab:2d} and \ref{tab:2d-ter}) shows that the period is necessarily $2$. In particular, we have $e_{2n+1} = 0$ for all $n\geq 0$. However, a reduction to the previous case is obtained by changing the step set into~$\mathfrak{S}+\mathfrak{S}$ and $n$ into~$n/2$.
\end{proof}

\subsection{Algorithmic irrationality proof}
\label{subsec:algorithmic} 

Let $\mathfrak{S} \subseteq \{0,\pm 1 \}^2$ be one of the 51 nonsingular step
sets with infinite group (see Table~\ref{tab:2d} in
Appendix~\ref{sec:appendix}). By Theorem~\ref{theo:proba}, the singular
exponent $\alpha$ in the asymptotic expansion of the excursion sequence
$(e^{\mathfrak{S}}_n)_{n\geq 0}$ is equal to $-1 - \pi/\arccos(-c)$, where $c$
is an algebraic number. Therefore, if $\arccos(c)/\pi$ is an
irrational number, then by Theorem~\ref{theo:arithmetic}, the generating
function $F_{\mathfrak{S}}(0,0,t)$ is not D-finite.

We now explain how, starting from the step set $\mathfrak{S}$ one can
\emph{algorithmically} prove that $\arccos(c)/\pi$ is irrational. This
effective proof decomposes into two main steps, solved by two different
algorithms. The first algorithm computes the minimal polynomial $\mu_c(t) \in
\mathbb{Q}[t]$ of~$c$ starting from $\mathfrak{S}$. The second one
performs computations on $\mu_c(t)$ showing that $\arccos(c)/\pi$ is
irrational.

\subsubsection{Computing the minimal polynomial of the correlation coefficient} Given $\chi = \chi_{\mathfrak{S}}$ the characteristic
polynomial of the step set~$\mathfrak{S}$, Theorem~\ref{theo:proba} shows that the exponential growth~$\rho$ and the correlation coefficient~$c$ are algebraic numbers, for which equations can be obtained by eliminating~$x$ and~$y$ from the algebraic equations
\[\frac{\partial\chi}{\partial x}=0,\qquad
\frac{\partial\chi}{\partial y}=0,\qquad
\rho-\chi=0,\qquad
c^2-\frac{\left({\frac{\Dx^2 \chi}{\Dx x
\Dx y}}\right)^2}{\frac{\Dx^2 \chi}{\Dx x^2} \cdot \frac{\Dx^2 \chi}{\Dx
y^2}}=0.
\]

This elimination is a routine task in effective algebraic geometry, usually performed with Gr\"obner bases for lexicographic or elimination orders~\cite{CoxLittleOShea1997}. 
Let $\chi_x$ and $\chi_y$ denote respectively the numerators of $\frac{\Dx \chi}{\Dx x}$ and of $\frac{\Dx \chi}{\Dx y}$. The solutions of~$\chi_x=\chi_y=0$ contain the solutions of~$\frac{\Dx \chi}{\Dx x}=\frac{\Dx \chi}{\Dx y}=0$, but may also contain spurious solutions at~$x=0$ or~$y=0$ provoked by the multiplication by powers of~$x$ and~$y$. These are removed by introducing a new variable~$u$ and considering the zero-dimensional ideal $I$ of~$\Q[x,y,u]$ generated by $(\chi_x,\chi_y,1-uxy)$.

For any zero~$(x_0,y_0)$ of the system~$\chi_x(x_0,y_0)=\chi_y(x_0,y_0)=0$ and any polynomials~$P(x,y)$ and~$Q(x,y)$ such that~$Q\not\in I$, the algebraic number~$P(x_0,y_0)/Q(x_0,y_0)$ is a root of a generator of the ideal~$I+\langle P(x,y)-tQ(x,y)\rangle\cap\mathbb{Q}[t]$. This can be used to compute annihilating polynomials for~$\rho$ and~$c$.

This computation is summarized in the following algorithm. 
\bigskip
\begin{itemize}
\item[{\sf Input:}] A step set $\mathfrak{S}$ satisfying the assumptions of Theorem~\ref{theo:proba}
\item[{\sf Output:}] The minimal polynomials of $\rho$ and $c$ defined in Theorem~\ref{theo:proba}
\item[]

\item[(1)] Set $\chi(x,y):= \sum_{(i,j)\in\mathfrak{S}}x^{i}y^{j}$, and
compute $\chi_x:= \textsf{numer}(\frac{\Dx \chi}{\Dx x})$, $\chi_y:=
\textsf{numer}(\frac{\Dx \chi}{\Dx y})$. 

\item[(2)] Compute the Gr\"obner basis of the ideal generated in $\mathbb{Q}[x,y,t,u]$ by $(\chi_x,\chi_y, \textsf{numer}(t-\chi),1-uxy)$  for a term order that eliminates~$x,y$ and~$u$. Isolate the unique polynomial in this basis that depends only on~$t$,
factor it, and identify its factor
$\mu_\rho$ that annihilates $\rho$.

\item[(3)] Compute the polynomial
\[P(x,y,t):=\textsf{numer}\left(t^2-\frac{\left({\frac{\Dx^2 \chi}{\Dx x
\Dx y}}\right)^2}{{\frac{\Dx^2 \chi}{\Dx x^2} \cdot \frac{\Dx^2 \chi}{\Dx
y^2}}}\right)\]
and eliminate~$x,y$ and~$u$ by computing a Gr\"obner basis of the ideal generated in $\mathbb{Q}[x,y,t,u]$ by~$(\chi_x,\chi_y,P,1-uxy)$ for a term order that eliminates~$x,y$ and $u$. Isolate the unique polynomial in this basis that depends only on~$t$,
factor it, and identify its factor
$\mu_c$ that annihilates $c$.

\end{itemize}
The identification of the proper factor is achieved for instance by computing a rough approximation of the numerical values of~$\rho$ and~$c$ using Eq.~\eqref{eq:definition-c} and comparing with the numerical roots of the factors. If needed, the numerical values can then be refined to arbitrary precision using the correct factor.

Table~\ref{tab:2d-ter} in Appendix~\ref{sec:appendix} displays the minimal
polynomials of $\rho$ and of $c$ obtained using this algorithm.

\subsubsection{Proving that the arccosine of the correlation coefficient is not commensurable with $\pi$}
Given the minimal polynomial~$\mu_c$ of the correlation coefficient~$c$, we now want to check that $\arccos(c)/\pi$ is irrational. General classification results exist, \eg,~\cite{Varona06}, but they are not
sufficient for our purpose. Instead, we
rather prove that $\arccos(c)/\pi$ is irrational in an algorithmic way. This
is based on the observation that if $\arccos(c)/\pi$ were rational, then~$c$ would be of the form {$(z+1/z)/2=(z^2+1)/(2z)$ with~$z$} a root of unity. This implies that
the
numerator of the rational function $\mu_c(\frac{x^2+1}{2x})$ would possess a
root which is a root of unity. In other words, the polynomial $R(x) = x^{\deg
\mu_c} \mu_c(\frac{x^2+1}{2x})$ would be divisible by a cyclotomic polynomial. This possibility can be
discarded by analyzing the minimal polynomials $\mu_c$ displayed in
Table~\ref{tab:2d-ter} in Appendix~\ref{sec:appendix}.

Indeed, in all the 51 cases, the polynomial $R(x)$ is irreducible and has
degree $2\deg(\mu_c)$, thus at most 28. Now, it is known that if the
cyclotomic polynomial $\Phi_N$ has degree at most 30, then $N$ is at most
150~\cite[Theorem~15]{RoSc62}, and the coefficients of $\Phi_N$ belong to the set
$\{-2,-1,0,1,2 \}$~\cite{Lehmer36}. Computing~$R$ in the 51 cases shows that
it has at least one coefficient of absolute value at least~3. This allows
to conclude that $R$ is not a cyclotomic polynomial, and therefore that
$\arccos(c)/\pi$ is irrational, and finishes the proof of Theorem~\ref{theo:main}.

\subsection{Example}
We now illustrate the systematic nature of our algorithms on Example~23 of Table~\ref{tab:2d}, \ie, walks with step set $\mathfrak S= \{(-1,0),(0,1),(1,0),(1,-1),(0,-1)\}$. For ease of use, we give explicit Maple instructions.

\subsubsection*{Step~1} The characteristic polynomial of the step set is
\begin{verbatim}
S:=[[-1,0],[0,1],[1,0],[1,-1],[0,-1]]:
chi:=add(x^s[1]*y^s[2],s=S);
\end{verbatim}
\[\chi:=\frac1x+\frac1y+x+y+\frac{x}y,\]
whose derivatives have numerators
\begin{verbatim}
chi_x:=numer(diff(chi,x));chi_y:=numer(diff(chi,y));
\end{verbatim}
\begin{equation}\label{eq:system-example}
\chi_x:=x^2+x^2y-y,\qquad\chi_y:=y^2-x-1.
\end{equation}
These define the system~\eqref{eq:system}.

\subsubsection*{Step~2} We now compute a polynomial that vanishes at~$\rho=\chi(x_0,y_0)$ when~$(x_0,y_0)$ is a solution of~\eqref{eq:system-example}. To this aim, we eliminate~$x$, $y$ and~$u$ in~$\{\chi_x,\chi_y,\textsf{numer}(\chi)-t\textsf{denom}(\chi),1-uxy\}$ by a Gr\"obner basis computation using an elimination order with~$(x,y,u)>t$. In Maple, this is provided by the command
\begin{verbatim}
G:=Groebner[Basis]([chi_x,chi_y,numer(t-chi),1-u*x*y],lexdeg([x,y,u],[t])):
\end{verbatim}
which returns five polynomials, only one of which is free of~$x$ and~$y$, namely
\begin{verbatim}
p:=factor(op(remove(has,G,{x,y,u})));
\end{verbatim}
\[p:=(t+1)(t^3+t^2-18t-43).\]
In this case, since we know that~$\rho>0$, we do not need to compute a numerical approximation of it using Eq.~\eqref{eq:definition-c}, but identify its minimal polynomial directly as~$\mu_\rho=t^3+t^2-18t-43$, which gives the entry in Column~3 of Table~\ref{tab:2d-ter}. The numerical value for~$\rho$ in Table~\ref{tab:2d} is given by
\begin{verbatim}
fsolve(p,t,0..infinity);
\end{verbatim}
\[4.729031538.\]
Note that in this example, the introduction of the variable~$u$ and the polynomial~$1-uxy$ are unnecessary since neither~$x=0$ nor $y=0$ are solutions of~$\chi_x=\chi_y=0$. 
\subsubsection*{Step~3}
Next, we obtain a polynomial which vanishes at~$c$ by a very similar computation:
\begin{verbatim}
G:=Groebner[Basis]([numer(t^2-diff(chi,x,y)^2/diff(chi,x,x)/diff(chi,y,y)), 
                 chi_x,chi_y,1-x*y*u],lexdeg([x,y,u],[t]));
\end{verbatim}
Again, this command returns five polynomials, with one of them free of~$x$ and~$y$, namely
\begin{verbatim}
p:=factor(op(remove(has,G,{x,y,u})));
\end{verbatim}
\[p:=(4t^2+1)(8t^3+8t^2+6t+1)(8t^3-8t^2+6t-1).\]
This polynomial has only two real roots, $\pm c$. Since~$c<0$ {(to see this, it suffices to use the expression~\eqref{eq:definition-c} of $c$)}, we identify its minimal polynomial as~$\mu_c=8t^3+8t^2+6t+1$, which gives the entry in Column~4 of Table~\ref{tab:2d-ter}.
Again, the numerical value for~$\alpha$ in Table~\ref{tab:2d} is given by
\begin{verbatim}
mu_c:=8*t^3+8*t^2+6*t+1:
evalf(-1-Pi/arccos(-fsolve(mu_c,t)));
\end{verbatim}
\[-3.320191962.\]

\subsubsection*{Step~4.} To conclude, we compute the polynomial
\[R(x)=x^3\mu_c\left(\frac{x^2+1}{2x}\right)=x^6+2x^5+6x^4+5x^3+6x^2+2x+1.\]
This polynomial does not have any root that is a root of unity, since it is irreducible and not cyclotomic:
\begin{verbatim}
R:=numer(subs(t=(x^2+1)/x/2, mu_c));	
irreduc(R),numtheory[iscyclotomic](R,x);
\end{verbatim}
\[\text{\em true}, \text{\em false}\]
This completes the proof that the generating function for this walk is not D-finite.

\section{Conclusion}
\subsection{Extensions}
The result of Denisov and Wachtel that is the basis of our work holds in arbitrary dimension and for walks with steps of arbitrary length.
The consequence that we state (Theorem~\ref{theo:proba}) is actually not  restricted to small step walks. Preliminary experiments indicate that it serves as a very efficient filter in a quest for holonomic excursion sequences.

In higher dimension, Theorem~\ref{theo:proba} no longer holds. The results of Denisov and Wachtel lead to a similar statement involving the first eigenvalue of a Laplacian on a spherical domain, which is difficult to estimate, not to mention irrationality proofs.

On the other hand, the arithmetic result of Theorem~\ref{theo:arithmetic} has a much wider scope of application in proving non-D-finiteness of combinatorial sequences and deserves to be better known.

\subsection{Open problems}
Since our approach brings together a strong arithmetic result (Theorem~\ref{theo:arithmetic}) and a strong probabilistic result (Theorem~\ref{theo:proba}), it appears natural to search for alternative simpler proofs of these results.

\subsubsection*{Proving that $\alpha$ is transcendental}
In Section \ref{subsec:algorithmic}, we are able to prove that for the 51 nonsingular models, the exponent $\alpha$ in the asymptotic expansion of the excursion sequence is irrational. It is worth mentioning that if it were possible to prove that $\alpha$ is not only irrational, but also transcendental, then Theorem~\ref{theo:arithmetic} would not be needed.

\subsubsection*{Simpler proof of Theorem~\ref{theo:arithmetic}}
The current proof of Theorem~\ref{theo:arithmetic} is based on several strong results from arithmetic theory \cite{ChCh85,Andre89,Katz70,DwGeSu94}. It would be interesting to know whether Theorem~\ref{theo:arithmetic} admits a simpler, direct proof.

\subsubsection*{Combinatorial proof of Theorem~\ref{theo:proba}} Theorem~\ref{theo:proba} relies on properties of the Brownian motion that are inherently continuous. Finding a purely combinatorial proof, \eg, with generating functions, would shed interesting light on this problem.

\subsubsection*{Total number of walks} Our work deals with the nature of the full generating function~$F_{\mathfrak S}(x,y,t)$ and its specialization at  $(x,y)=(0,0)$. 
The actual nature of~$F_{\mathfrak S}(1,1,t)$ is still partly unknown at this stage, although the conjectural classification in~\cite{BoKa09,BoKa} suggests that $F_{\mathfrak S}(1,1,t)$ is not-D-finite in all 51 non-singular cases with infinite group.
Another approach might be needed in this problem. Indeed, when $\mathfrak S=\{(1,0),(0,1),(1,1),(-1,-1)\}$, the total number of walks of length~$n$ seems to behave like~$c\cdot 4^n$, but the generating function itself seems to be non-D-finite. 

\medskip\noindent{\bf Acknowledgements.} We wish to thank Tanguy Rivoal, Denis Denisov and Vitali Wachtel for stimulating exchanges. We also thank the referees for their comments. Work of the first and the third authors was supported in part by the Microsoft Research-Inria Joint Centre.

 \bibliographystyle{unsrt}
 \makeatletter
 \def\@openbib@code{\itemsep=-3pt}
 \makeatother

\def\cprime{$'$}

 \clearpage
 \section{Appendix}\label{sec:appendix}

\normalsize
\mytable\label{tab:2d} The 51 nonsingular step sets in the quarter plane with an infinite group, and the asymptotics of their excursions 
(valid for even $n$ in the periodic cases).
The numbering used in column ``Tag'' corresponds to the order of the step sets in
Table~4 of~\cite{BoMi10}. 
Periodic walks have tags marked with a star.

\medskip
\def\itemA#1#2#3#4#5{#2 & #3 & #4 &\rule[-.5em]{0pt}{2.35em}\kern-1.2pt$#5$\kern-1.2pt&\global\let\item\itemB}
\def\itemB#1#2#3#4#5{#2 & #3 & #4 &\kern-1.2pt$#5$\kern-1.2pt\\\global\let\item\itemA}
\let\item\itemA
\footnotesize
\begin{longtable}{@{}c|c|c|c||c|c|c|c@{}}
  Tag & Steps & First terms & Asymptotics & Tag & Steps & First terms & Asymptotics \\\hline \endhead
 \item{3}{3}{\stepset11001010}{1, 0, 1, 2, 2, 13, 21, 67, 231}{\dfrac{3.799605^n}{n^{2.610604}}}%
 \item{4}{4}{\stepset10101010}{1, 0, 0, 2, 2, 0, 16, 44, 28}{\dfrac{ 3.608079^n}{n^{2.720448}}}%
 \item{5}{5}{\stepset10001011}{1, 0, 1, 2, 2, 14, 21, 76, 252}{\dfrac{3.799605^n}{n^{2.318862}}}%
 \item{6}{6}{\stepset01010011}{1, 0, 1, 2, 2, 13, 21, 67, 231}{\dfrac{3.799605^n}{n^{2.610604}}}%
 \item{7}{7$^\star$}{\stepset10100011}{1, 0, 1, 0, 4, 0, 29, 0, 230}{\dfrac{3.800378^{n}}{n^{2.521116}}}%
 \item{8}{8}{\stepset00111010}{1, 0, 1, 1, 2, 7, 10, 38, 89}{\dfrac{3.799605^n}{n^{3.637724}}}%
 \item{9}{9}{\stepset01110010}{1, 0, 1, 1, 2, 7, 10, 38, 89 }{\dfrac{3.799605^n}{n^{3.637724}}}%
 \item{10}{10}{\stepset10110010}{1, 0, 0, 1, 2, 0, 5, 26, 28}{\dfrac{3.608079^n}{n^{3.388025}}}%
 \item{11}{11$^\star$}{\stepset10100110}{1, 0, 0, 0, 2, 0, 6, 0, 42}{\dfrac{3.800378^{n}}{n^{3.918957}}}%
 \item{12}{12}{\stepset00110110}{1, 0, 0, 1, 0, 1, 5, 1, 18}{\dfrac{3.799605^n}{n^{5.136154}}}%
 \item{14}{14}{\stepset00110011}{1, 0, 0, 1, 2, 0, 5, 26, 28 }{\dfrac{3.608079^n}{n^{3.388025}}}%
 \item{16}{16}{\stepset11010001}{1, 0, 1, 2, 2, 14, 21, 76, 252}{\dfrac{3.799605^n}{n^{2.318862}}}%
 \item{17}{17$^\star$}{\stepset11000101}{1, 0, 1, 0, 4, 0, 29, 0, 230}{\dfrac{3.800378^{n}}{n^{2.521116}}}%
 \item{18}{18}{\stepset01010101}{1, 0, 0, 2, 2, 0, 16, 44, 28}{\dfrac{ 3.608079^n}{n^{2.720448}}}%
 \item{19}{19$^\star$}{\stepset01100101}{1, 0, 0, 0, 2, 0, 6, 0, 42}{\dfrac{3.800378^{n}}{n^{3.918957}}}%
\item{20}{20}{\stepset10001111}{1, 0, 1, 2, 4, 14, 45, 120, 468}{\dfrac{4.372923^n}{n^{2.482876}}}%
 \item{21}{21}{\stepset01001111}{1, 0, 1, 1, 4, 7, 25, 64, 201}{\dfrac{4.214757^n}{n^{3.347502}}}%
 \item{23}{23}{\stepset01111010}{1, 0, 2, 1, 10, 14, 75, 178, 738}{\dfrac{4.729032^n}{n^{3.320192}}}%
 \item{24}{24}{\stepset11011010}{1, 0, 2, 2, 10, 26, 86, 312, 1022}{\dfrac{4.729032^n}{n^{2.757466}}}%
 \item{25}{25}{\stepset10011011}{1, 0, 2, 2, 11, 27, 101, 348, 1237}{\dfrac{4.729032^n}{n^{2.397625}}}%
 \item{26}{26}{\stepset11010011}{1, 0, 2, 2, 11, 27, 101, 348, 1237}{\dfrac{4.729032^n}{n^{2.397625}}}%
 \item{27}{27$^\star$}{\stepset11000111}{1, 0, 2, 0, 13, 0, 124, 0, 1427}{\dfrac{4.569086^{n}}{n^{2.503534}}}%
 \item{28}{28}{\stepset01010111}{1, 0, 1, 2, 4, 13, 36, 111, 343}{\dfrac{4.214757^n}{n^{2.742114}}}%
 \item{29}{29$^\star$}{\stepset01100111}{1, 0, 1, 0, 5, 0, 35, 0, 313}{\dfrac{4.569086^{n}}{n^{3.985964}}}%
 \item{30}{30}{\stepset10110011}{1, 0, 1, 1, 6, 17, 58, 202, 749}{\dfrac{5^n}{n^{2.722859}}}%
 \item{31}{31}{\stepset00110111}{1, 0, 0, 1, 2, 1, 11, 27, 60}{\dfrac{4.372923^n}{n^{4.070925}}}%
 \item{32}{32$^\star$}{\stepset11100011}{1, 0, 2, 0, 13, 0, 124, 0, 1427}{\dfrac{4.569086^{n}}{n^{2.503534}}}%
 \item{33}{33}{\stepset11110010}{1, 0, 1, 1, 4, 7, 25, 64, 201}{\dfrac{4.214757^n}{n^{3.347502}}}%
 \item{34}{34$^\star$}{\stepset11100110}{1, 0, 1, 0, 5, 0, 35, 0, 313}{\dfrac{4.569086^{n}}{n^{3.985964}}}%
 \item{35}{35}{\stepset01110110}{1, 0, 1, 1, 3, 8, 19, 65, 177}{\dfrac{4.729032^n}{n^{4.514931}}}%
 \item{36}{36}{\stepset10110110}{1, 0, 0, 1, 2, 1, 11, 27, 60}{\dfrac{4.372923^n}{n^{4.070925}}}%
 \item{37}{37}{\stepset11101010}{1, 0, 1, 2, 4, 13, 36, 111, 343}{\dfrac{4.214757^n}{n^{2.742114}}}%
 \item{38}{38}{\stepset01011011}{1, 0, 2, 2, 10, 26, 86, 312, 1022}{\dfrac{4.729032^n}{n^{2.757466}}}%
 \item{39}{39}{\stepset01101110}{1, 0, 1, 1, 3, 8, 19, 65, 177}{\dfrac{4.729032^n}{n^{4.514931}}}%
 \item{40}{40}{\stepset10101110}{1, 0, 0, 2, 4, 8, 28, 108, 372}{\dfrac{5^n}{n^{3.383396}}}%
 \item{41}{41}{\stepset11010101}{1, 0, 1, 2, 4, 14, 45, 120, 468}{\dfrac{4.372923^n}{n^{2.482876}}}%
 \item{42}{42}{\stepset01110101}{1, 0, 0, 2, 4, 8, 28, 108, 372}{\dfrac{5^n}{n^{3.383396}}}%
 \item{43}{43}{\stepset11101011}{1, 0, 2, 2, 13, 27, 140, 392, 1882}{\dfrac{5.064419^n}{n^{2.491053}}}%
 \item{44}{44}{\stepset10111011}{1, 0, 2, 3, 15, 51, 208, 893, 3841}{\dfrac{5.891838^n}{n^{2.679783}}}%
 \item{45}{45}{\stepset00111111}{1, 0, 1, 1, 5, 8, 40, 91, 406}{\dfrac{5.064419^n}{n^{4.036441}}}%
 \item{46}{46}{\stepset10101111}{1, 0, 1, 2, 8, 22, 101, 364, 1618}{\dfrac{5.799605^n}{n^{2.959600}}}%
 \item{47}{47}{\stepset10111110}{1, 0, 1, 3, 7, 29, 101, 404, 1657}{\dfrac{5.891838^n}{n^{3.471058}}}%
 \item{48}{48}{\stepset11110110}{1, 0, 1, 1, 5, 8, 40, 91, 406}{\dfrac{5.064419^n}{n^{4.036441}}}%
 \item{49}{49}{\stepset11010111}{1, 0, 2, 2, 13, 27, 140, 392, 1882}{\dfrac{5.064419^n}{n^{2.491053}}}%
 \item{50}{50}{\stepset11110011}{1, 0, 2, 3, 15, 51, 208, 893, 3841}{\dfrac{5.891838^n}{n^{2.679783}}}%
 \item{51}{51}{\stepset01110111}{1, 0, 1, 3, 7, 29, 101, 404, 1657}{\dfrac{5.891838^n}{n^{3.471058}}}%
 \item{52}{52}{\stepset10110111}{1, 0, 1, 1, 8, 18, 90, 301, 1413}{\dfrac{5.799605^n}{n^{3.042101}}}%
 \item{53}{53}{\stepset11110101}{1, 0, 1, 2, 8, 22, 101, 364, 1618}{\dfrac{5.799605^n}{n^{2.959600}}}%
 \item{54}{54}{\stepset11011111}{1, 0, 3, 5, 30, 111, 548, 2586, 13087}{\dfrac{6.729032^n}{n^{2.667986}}}%
 \item{55}{55}{\stepset11111110}{1, 0, 2, 4, 16, 64, 266, 1210, 5630}{\dfrac{6.729032^n}{n^{3.497037}}}%
 \item{56}{56}{\stepset01111111}{1, 0, 2, 4, 16, 64, 266, 1210, 5630}{\dfrac{6.729032^n}{n^{3.497037}}}%
\end{longtable}%

\normalsize
\mytable\label{tab:2d-ter} 
Minimal polynomials of the growth constants~$\rho$ and of the correlation coefficients~$c$ for the 51 nonsingular walks in the quarter plane with an infinite group. Each blank entry in the table coincides with the first non-empty entry above it.

\medskip

\def\item#1#2#3#4#5#6#7{#2 & #3 &
$#6$ & $#7$
  \\}
\scriptsize
\begin{longtable}{@{}c|c|c|c@{}}
  Tag & Steps  & Minimal polynomial $\mu_\rho$ of $\rho$ & Minimal polynomial $\mu_c$ of $c=-\cos(\frac{\pi}{1+\alpha})$   \\\hline \endhead
 \item{12}{12}{\stepset00110110}{1, 0, 0, 1, 0, 1, 5, 1, 18}{\dfrac{3.799605^n}{n^{5.136154}}}{t^4+t^3-8t^2-36t-11 }{t^4+\frac92t^3+\frac{27}{4}t^2+\frac{35}{8}t+\frac{17}{16}}%
 \item{(5, 16)}{5, 16}{\stepset10001011, \stepset11010001}{1, 0, 1, 2, 2, 14, 21, 76, 252}{\dfrac{3.799605^n}{n^{2.318862}}}{  }{t^4-\frac92t^3+\frac{27}{4}t^2-\frac{35}{8}t+\frac{17}{16}}%
 \item{(3,6)}{3, 6}{\stepset11001010, \stepset01010011}{1, 0, 1, 2, 2, 13, 21, 67, 231}{\dfrac{3.799605^n}{n^{2.610604}}}{}{t^8+\frac14t^6-\frac{3}{16}t^4+\frac{3}{64}t^2-\frac{1}{256}}%
 \item{(8, 9)}{8, 9}{\stepset00111010, \stepset01110010}{1, 0, 1, 1, 2, 7, 10, 38, 89}{\dfrac{3.799605^n}{n^{3.637724}}}{  }{  }%

\hline
 \item{(7, 17)}{7$^\star$, 17$^\star$}{\stepset10100011, \stepset11000101}{1, 0, 1, 0, 4, 0, 29, 0, 230}{\dfrac{3.800378^n}{n^{2.521116}}}{t^6-11t^4-32t^2-256}{t^6+\frac34 t^4+2t^2-\frac12}%
 \item{(11, 19)}{11$^\star$, 19$^\star$}{\stepset10100110, \stepset01100101}{1, 0, 0, 0, 2, 0, 6, 0, 42}{\dfrac{3.800378^n}{n^{3.918957}}}{  }{  }%

\hline
 \item{(4,18)}{4, 18}{\stepset10101010, \stepset01010101}{1, 0, 0, 2, 2, 0, 16, 44, 28}{\dfrac{ 3.608079^n}{n^{2.720448}}}{t^5+t^4+t^3-30t^2-96t-91}{t^{10}+2t^8+t^6-\frac{1}{64}t^4+\frac{3}{256}t^2-\frac{1}{1024}}%
 \item{(10, 14)}{10, 14}{\stepset10110010, \stepset00110011}{1, 0, 0, 1, 2, 0, 5, 26, 28}{\dfrac{3.608079^n}{n^{3.388025}}}{  }{  }%

\hline
\item{(20, 41)}{20, 41}{\stepset10001111, \stepset11010101}{1, 0, 1, 2, 4, 14, 45, 120, 468}{\dfrac{4.372923^n}{n^{2.482876}}}{t^5-2t^4-4t^3-31t^2+23t-41}{t^{10}+t^8+\frac{157}{32}t^6+\frac{145}{128}t^4+\frac{1681}{512}t^2-\frac{2209}{2048}}
\item{(31, 36)}{31, 36}{\stepset00110111, \stepset10110110}{1, 0, 0, 1, 2, 1, 11, 27, 60}{\dfrac{4.372923^n}{n^{4.070925}}}{  }{  }%

\hline
 \item{(21, 33)}{21, 33}{\stepset01001111, \stepset11110010}{1, 0, 1, 1, 4, 7, 25, 64, 201}{\dfrac{4.214757^n}{n^{3.347502}}}{t^5+2t^4-7t^3-46t^2-116t-131}{t^{10}+\frac32 t^8+\frac{13}{16}t^6+\frac{5}{64}t^4+\frac{3}{256}t^2-\frac{1}{1024}}
 \item{(28, 37)}{28, 37}{\stepset01010111, \stepset11101010}{1, 0, 1, 2, 4, 13, 36, 111, 343}{\dfrac{4.214757^n}{n^{2.742114}}}{  }{  }%

\hline
 \item{23}{23}{\stepset01111010}{1, 0, 2, 1, 10, 14, 75, 178, 738}{\dfrac{4.729032^n}{n^{3.320192}}}{t^3+t^2-18t-43}{t^3+t^2+\frac34t+\frac18}%
 \item{(24, 38)}{24, 38}{\stepset11011010, \stepset01011011}{1, 0, 2, 2, 10, 26, 86, 312, 1022}{\dfrac{4.729032^n}{n^{2.757466}}}{  }{t^3-t^2+\frac34t-\frac18}%
 \item{(25, 26)}{25, 26}{\stepset10011011, \stepset11010011}{1, 0, 2, 2, 11, 27, 101, 348, 1237}{\dfrac{4.729032^n}{n^{2.397625}}}{  }{t^6-t^4+\frac{7}{16} t^2-\frac{5}{64}}%
 \item{(35, 39)}{35, 39}{\stepset01110110, \stepset01101110}{1, 0, 1, 1, 3, 8, 19, 65, 177}{\dfrac{4.729032^n}{n^{4.514931}}}{  }{  }%

\hline
 \item{(27, 32)}{27$^\star$, 32$^\star$}{\stepset11000111, \stepset11100011}{1, 0, 2, 0, 13, 0, 124, 0, 1427}{\dfrac{4.569086^n}{n^{2.503534}}}{t^6-20t^4-16t^2-48}{t^6+2t^4+\frac52t^2-\frac34}%
 \item{(29, 34)}{29$^\star$, 34$^\star$}{\stepset01100111, \stepset11100110}{1, 0, 1, 0, 5, 0, 35, 0, 313}{\dfrac{4.569086^n}{n^{3.985964}}}{  }{  }%

\hline
 \item{30}{30}{\stepset10110011}{1, 0, 1, 1, 6, 17, 58, 202, 749}{\dfrac{5^n}{n^{2.722859}}}{t-5}{t-\frac14}%
 \item{(40, 42)}{40, 42}{\stepset10101110, \stepset01110101}{1, 0, 0, 2, 4, 8, 28, 108, 372}{\dfrac{5^n}{n^{3.383396}}}{  }{t+\frac14}%

\hline
 \item{(43, 49)}{43, 49}{\stepset11101011, \stepset11010111}{1, 0, 2, 2, 13, 27, 140, 392, 1882}{\dfrac{5.064419^n}{n^{2.491053}}}{t^6+2t^5-18t^4-67t^3-108t^2-40t-19}{t^{12}+\frac{11}{4}t^{10}+\frac{107}{16}t^8+\frac{145}{32}t^6+\frac{455}{128}t^4-\frac{2859}{1024}t^2+\frac{1521}{4096}}%
 \item{(45, 48)}{45, 48}{\stepset00111111, \stepset11110110}{1, 0, 1, 1, 5, 8, 40, 91, 406}{\dfrac{5.064419^n}{n^{4.036441}}}{  }{  }%

\hline
 \item{(44, 50)}{44, 50}{\stepset10111011, \stepset11110011}{1, 0, 2, 3, 15, 51, 208, 893, 3841}{\dfrac{5.891838^n}{n^{2.679783}}}{t^7+3t^6-18t^5-127t^4-328t^3-560t^2-704t-448}{t^{14}+\frac{23}{4}t^{12}+\frac{25}{2}t^{10}+\frac{971}{64}t^8+\frac{421}{32}t^6+\frac{307}{64}t^4+\frac{107}{64}t^2-\frac{49}{256}}%
 \item{(47, 51)}{47, 51}{\stepset10111110, \stepset01110111}{1, 0, 1, 3, 7, 29, 101, 404, 1657}{\dfrac{5.891838^n}{n^{3.471058}}}{  }{  }%

\hline
 \item{52}{52}{\stepset10110111}{1, 0, 1, 1, 8, 18, 90, 301, 1413}{\dfrac{5.799605^n}{n^{3.042101}}}{t^4-7t^3+10t^2-24t+37 }{t^4+\frac12t^3+\frac{55}{4}t^2+\frac{19}{8}t+\frac{1}{16}}%
 \item{(46, 53)}{46, 53}{\stepset10101111, \stepset11110101}{1, 0, 1, 2, 8, 22, 101, 364, 1618}{\dfrac{5.799605^n}{n^{2.959600}}}{ }{t^4-\frac12t^3+\frac{55}{4}t^2-\frac{19}{8}t+\frac{1}{16}}%

\hline
 \item{54}{54}{\stepset11011111}{1, 0, 3, 5, 30, 111, 548, 2586, 13087}{\dfrac{6.729032^n}{n^{2.667986}}}{t^3-5t^2-10t-11}{t^3+\frac{11}{4}t-\frac78}%
 \item{(55, 56)}{55, 56}{\stepset11111110, \stepset01111111}{1, 0, 2, 4, 16, 64, 266, 1210, 5630}{\dfrac{6.729032^n}{n^{3.497037}}}{  }{t^3+\frac{11}{4}t+\frac78}%
\end{longtable}%
\end{document}